\documentclass[12pt]{amsart}

\usepackage{amsfonts,amsmath,amsthm,amscd,amssymb,latexsym}
\usepackage[english]{babel}

\theoremstyle{plain}

\newtheorem{thm}{Theorem}
\newtheorem{lem}[thm]{Lemma}
\newtheorem{prop}[thm]{Proposition}
\newtheorem{cor}[thm]{Corollary}
\newtheorem{rmk}{Remark}
\newtheorem{dfn}{Definition}
\newtheorem{exm}{Example}

\newtheorem{thmm}{Theorem}

\begin{document}

\title{On Baire classification of mappings with values in connected spaces}

\author{Olena Karlova}

\email{maslenizza.ua@gmail.com}

\address{Department of Mathematical Analysis, Faculty of Mathematics and Informatics, Chernivtsi National University, Kotsyubyns'koho str., 2, Chernivtsi, Ukraine}

\begin{abstract}
  We generalize the Lebesgue-Hausdorff Theorem on the characterization of Baire-one functions for $\sigma$-strongly functionally discrete mappings defined on arbitrary topological spaces.
\end{abstract}

\keywords{ $\sigma$-strongly functionally discrete mapping, Baire-one mapping, Lebesgue-Hausdorff theorem}

\subjclass{Primary 54C08, 26A21;  Secondary 54C50, 54H05}

\maketitle

\section{Introduction}

A subset $A$ of a topological space $X$ is  {\it functionally $F_\sigma$- ($G_\delta$)-set} if  $A$ is the union (the intersection) of a sequence of functionally closed  (functionally open) subsets of $X$. 

Let $X$ and $Y$ be topological spaces and $f:X\to Y$ be a mapping. We say that $f$ belongs to
\begin{itemize}
\item  {\it the first Baire class}, $f\in {\rm B}_1(X,Y)$, if $f$ is a pointwise limit of a sequence of continuous mappings between  $X$ and $Y$;

   \item {\it the first Lebesgue class}, $f\in {\rm H}_1(X,Y)$, if $f^{-1}(V)$ is an $F_\sigma$-set in $X$ for any open subset $V$ of  $Y$;

   \item {\it the first functional Lebesgue class}, $f\in {\rm K}_1(X,Y)$, if $f^{-1}(V)$ is a functionally  $F_\sigma$-set in $X$ for any open subset $V$ of  $Y$;

   \item {\it the first weak functional Lebesgue class}, $f\in {\rm K}_1^w(X,Y)$, if $f^{-1}(V)$ is a functionally $F_\sigma$-set in $X$ for any functionally open subset $V$ of  $Y$.
\end{itemize}

The classical Lebesgue-Hausdorff theorem~\cite{Lebesgue:1905, Hausdorff:1957} tells us that the classes ${\rm B}_1(X,Y)$ and ${\rm H}_1(X,Y)$ coincide if $X$ is  a metrizable space and $Y=[0,1]^{\omega}$. This theorem was generalized by many mathematicians (see~\cite{Rolewicz, Laczk, Stegall, Hansell:1992, Rogers, JOPV, Fos, Vesely} and the literature given there). For classification of mappings with values in non-separable metric spaces, R.~Hansell in~\cite{Hansell:1971} introduced a notion of $\sigma$-discrete mapping. Namely, a mapping $f:X\to Y$ is {\it $\sigma$-discrete} if there exists a family $\mathcal B=\bigcup\limits_{n=1}^\infty \mathcal B_n$ of subsets of $X$ such that every family $\mathcal B_n$ is discrete in $X$ and the preimage $f^{-1}(V)$ of any open set $V\subseteq Y$ is a union of sets from $\mathcal B$. The class of all $\sigma$-discrete mappings between $X$ and $Y$ is denoted by $\Sigma(X,Y)$. Among others we note the following generalizations of the Lebesgue-Hausdorff theorem:

\begin{thmm}[R. Hansell~\cite{Hansell:1992}]\label{th:Hansell} Let $X$ be a collectionwise normal space and $Y$ be a closed convex subset of a Banach space. Then
${\rm B}_1(X,Y)={\rm H}_1(X,Y)\cap\Sigma(X,Y)$.
\end{thmm}

\begin{thmm}[M.~Fosgerau~\cite{Fos}]\label{th:Fosgerau} Let  $Y$ be a complete metric space. Then the following conditions are equivalent:
\begin{enumerate}
  \item[{\rm (i)}] $Y$ is connected and locally connected;

  \item[{\rm (ii)}] ${\rm H}_1([0,1],Y)={\rm B}_1([0,1],Y)$;

  \item[{\rm (iii)}] ${\rm H}_1(X,Y)\cap\Sigma(X,Y)={\rm B}_1(X,Y)$ for any metric space $X$.
\end{enumerate}
\end{thmm}

In order to weaken the condition on $X$ in Theorem~\ref{th:Fosgerau}, L.~Vesel\'{y}~\cite{Vesely} considered strongly  $\sigma$-discrete mappings. A mapping $f:X\to Y$ is {\it strongly $\sigma$-discrete} if there exists a family $\mathcal B=\bigcup\limits_{n=1}^\infty\mathcal B_n$ of subsets of $X$ such that for every family  $\mathcal B_n=(B_i:i\in I_n)$ there exists a discrete family $(U_i:i\in I_n)$ of open sets in $X$ with $\overline{B_i}\subseteq U_i$ for all  $i\in I_n$ and, moreover, the preimage $f^{-1}(V)$ of any open set $V$ in $Y$ is a union of sets from $\mathcal B$. Vesel\'{y}~\cite{Vesely} denoted the collection of all such mappings by $\Sigma^*(X,Y)$ and proved the following theorem.

\begin{thmm}[L.~Vesel\'{y}~\cite{Vesely}]\label{th:Vesely}
Let $X$ be a normal space and $Y$ be a connected and locally arcwise connected metrizable  space. Then
${\rm B}_1(X,Y)={\rm H}_1(X,Y)\cap\Sigma^*(X,Y)$.
\end{thmm}

The aim of this paper is a generalization of Theorems~\ref{th:Fosgerau} and~\ref{th:Vesely} for mappings defined on an arbitrary topological space.
With this purpose in Section~\ref{sect:one} we shall introduce a class $\Sigma^f(X,Y)$ of $\sigma$-strongly functionally discrete mappings (which  coincides with the class $\Sigma^*(X,Y)$ if the space $X$ is normal).

\begin{thmm}
  Let $X$ be a topological space and $Y$ be a connected and locally arcwise connected metrizable space. Then
  ${\rm B}_1(X,Y)={\rm K}_1(X,Y)\cap\Sigma^f(X,Y)$.
\end{thmm}

\begin{thmm} For a completely metrizable space $Y$ the following conditions are equivalent:
\begin{enumerate}
   \item[{\rm (i)}] $Y$ is connected and locally arcwise connected;

   \item[{\rm (ii)}] ${\rm K}_1(X,Y)\cap \Sigma^f(X,Y)={\rm B}_1(X,Y)$ for any topological space $X$;

   \item[{\rm (iii)}] ${\rm H}_1([0,1],Y)={\rm B}_1([0,1],Y)$.
\end{enumerate}
\end{thmm}

\section{Relations between Baire-one, Lebesgue-one and $\sigma$-discrete mappings}\label{sect:one}

\begin{dfn}
 {\rm A family $\mathcal A=(A_i:i\in I)$ of subsets of a topological space $X$ is called
  \begin{enumerate}
    \item[(1)] {\it discrete}, if every point of $X$ has an open neighborhood which intersects at most one set from $\mathcal A$;

    \item[(2)] {\it strongly discrete}, if there exists a discrete family $(U_i:i\in I)$ of open subsets of $X$ such that $\overline{A_i}\subseteq U_i$ for every $i\in I$;

    \item[(3)] {\it strongly functionally discrete} or {\it sfd-family}, if there exists a discrete family $(U_i:i\in I)$ of functionally open subsets of $X$ such that $\overline{A_i}\subseteq U_i$ for every $i\in I$;
  \end{enumerate}}
\end{dfn}

Notice that   (3) $\Rightarrow$ (2) $\Rightarrow$ (1) for any $X$; a topological space $X$ is collectionwise normal if and only if every discrete family in $X$ is strongly discrete; if $X$ is normal then (2) $\Leftrightarrow$ (3).

\begin{dfn}
{\rm Let $\mathcal P$ be a property of a family of subsets of topological space. A family $\mathcal A$ is called {\it a $\sigma$-$\mathcal P$-family} if  $\mathcal A=\bigcup\limits_{n=1}^\infty {\mathcal A}_n$, where every family $\mathcal A_n$ has the property $\mathcal P$.}
\end{dfn}

\begin{dfn}\label{def:base-for-a-function}
{\rm A family $\mathcal B$ of sets of a topological space $X$ is called {\it a base}  for a mapping $f:X\to Y$ if the preimage $f^{-1}(V)$ of an arbitrary open set  $V$ in $Y$ is a union of sets from $\mathcal B$.}
\end{dfn}

\begin{dfn}
{\rm  If a mapping $f:X\to Y$ has a base which is a $\sigma$-$\mathcal P$-family, then we say that $f$ is {\it a $\sigma$-$\mathcal P$ mapping}.}
\end{dfn}

The collection of all $\sigma$-$\mathcal P$ mappings between $X$ and $Y$  will be denoted by
\begin{itemize}
  \item $\Sigma(X,Y)$ if $\mathcal P$ is the property of discreteness;

  \item $\Sigma^*(X,Y)$ if $\mathcal P$ is the property of a strong discreteness;

  \item $\Sigma^f(X,Y)$ if $\mathcal P$ is the property of a strong functional discreteness.
\end{itemize}

By $\Sigma^{f*}_{0}(X,Y)$ we denote the collection of all mappings between $X$ and $Y$ which has a $\sigma$-sfd base of functionally closed subsets of $X$.

Notice that every mapping into a second-countable space is $\sigma$-discrete. Hansell proved in~\cite{Hansell:1971} that  every Borel measurable mapping (in particular, every ${\rm H}_1$-mapping) $f:X\to Y$ is $\sigma$-discrete if $X$ is a complete metric space and $Y$ is a metric space.

\begin{thm}\label{BaireOne_is_SFD}
  Let $X$ be a topological space and $Y$ be a metrizable space. Then
  $$
  {\rm B}_1(X,Y)\subseteq \Sigma^{f*}_0(X,Y)={\rm K}_1(X,Y)\cap\Sigma^f(X,Y).
  $$
\end{thm}

\begin{proof}
The equality $\Sigma^{f*}_0(X,Y)={\rm K}_1(X,Y)\cap\Sigma^f(X,Y)$ is proved in~\cite[Theorem 6]{Karlova:TA:2015}.

Let $f\in {\rm B}_1(X,Y)$ and $(f_n)_{n=1}^\infty$ be a sequence of continuous mappings $f_n:X\to Y$ such that $f(x)=\lim\limits_{n\to\infty} f_n(x)$ for every $x\in X$. Fix any $\sigma$-discrete open base $\mathcal V=\bigcup\limits_{m=1}^\infty \mathcal V_m$ of  $Y$, where each family  $\mathcal V_m$ is discrete. For every $V\in \mathcal V$ we take a sequence $(G_{k,V})_{k=1}^\infty$ of open sets such that $\overline{G_{k,V}}\subseteq G_{k+1,V}$ for every $k\in\mathbb N$ and $V=\bigcup\limits_{k=1}^\infty \overline{G_{k,V}}$. Observe that
  \begin{gather*}
  f^{-1}(V)=\bigcup\limits_{k=1}^\infty\bigcap\limits_{n=k}^{\infty} f_n^{-1}(\overline{G_{k,V}}).
  \end{gather*}
Denote $F_{k,V}=\bigcap\limits_{n=k}^{\infty} f_n^{-1}(\overline{G_{k,V}})$ and notice that every set $F_{k,V}$ is functionally closed in $X$. For all $k,m\in\mathbb N$ we put $\mathcal B_{k,m}=(F_{k,V}:V\in\mathcal V_m)$ and  $\mathcal B=\bigcup\limits_{k,m=1}^\infty\mathcal B_{k,m}$. Then $\mathcal B$ is a base for $f$. Moreover, every family $\mathcal B_{k,m}$ is strongly functionally discrete, since $F_{k,V}\subseteq f_k^{-1}(V)$ and the family $(f_k^{-1}(V):V\in\mathcal V_m)$ is discrete and consists of functionally open sets. Hence, $\mathcal B$ is a $\sigma$-sfd base for $f$. Thus, $f\in \Sigma^{f*}_0(X,Y)$.
\end{proof}

\begin{prop}\label{prop:H1=B1}
Let $X$ and $Y$ be topological spaces. Then
\begin{enumerate}
  \item ${\rm K}_1(X,Y)\subseteq {\rm K}_1^w(X,Y)$;

  \item ${\rm K}_1(X,Y)={\rm H}_1(X,Y)$ if $X$ is perfectly normal;

  \item ${\rm K}_1(X,Y)={\rm K}_1^w(X,Y)\subseteq {\rm H}_1(X,Y)$ if $Y$ is perfectly normal;

  \item ${\rm H}_1(X,Y)\subseteq{\rm K}_1^w(X,Y)$ if $X$ is normal;

  \item ${\rm H}_1(\upsilon X,Y)\subseteq {\rm K}_1^w(\upsilon X,Y)$ if $X$ is normal.
\end{enumerate}
\end{prop}

\begin{proof} Since (1), (2) and (3) are obvious, we prove (4) and (5).

Let $F$ be a functionally closed subset of $Y$ and $\varphi:Y\to [0,1]$ be a continuous function with $F=\varphi^{-1}(0)$. We  put $G_n=\varphi^{-1}([0,\frac 1n))$ for every $n\in\mathbb N$.

  (4) Let $f\in {\rm H}_1(X,Y)$. For every $n$ the sets  $A_n=f^{-1}(G_n)$ and $B_n=X\setminus f^{-1}(\overline{G}_n)$ are disjoint and  $F_\sigma$ in $X$. Then by~\cite[Lemma 1.7]{Vesely} for every  $n$ there exists a functionally $G_\delta$-set $C_n$ in $X$ such that $A_n\subseteq C_n$ and $C_n\cap B_n=\emptyset$. Notice that $f^{-1}(F)=\bigcap\limits_{n=1}^\infty C_n$. Hence,  $f^{-1}(F)$ is a functionally $G_\delta$-set in $X$. Therefore, $f\in {\rm K}_1^w(X,Y)$.

  (5) Let $f\in {\rm H}_1(\upsilon X,Y)$. Put
  \begin{gather*}
  A=f^{-1}(F), \,\,\,A_n=f^{-1}(G_{n+2}), \\
  B_n=\upsilon X\setminus f^{-1}(G_n), \,\,\, C_n=\upsilon X\setminus f^{-1}(\overline{G}_{n+1}).
    \end{gather*}
  Then
$A\subseteq A_n$, $B_n\subseteq C_n$, $A_n\cap C_n=\emptyset$, $A$ and $B_n$ are $G_\delta$-sets in  $\upsilon X$, $A_n$ and $C_n$ are $F_\sigma$-sets in $\upsilon X$. If $A=\emptyset$ or $\upsilon X\setminus A=\emptyset$, then $A$ is a functionally $G_\delta$-set. Thus, it is sufficient to consider the case of $A\ne\emptyset$ and $\upsilon X\setminus A\ne\emptyset$. Moreover, we may assume that $B_n\ne\emptyset$ for all $n\in\mathbb N$.
Since $X$ intersects every nonempty $G_\delta$-subset of $\upsilon X$ (and, consequently, every nonempty $G_{\delta\sigma}$-set) by~\cite[Corollary 2.6]{Blair-Hager},  $A_X=A\cap X\ne\emptyset$ and $B_{n,X}=B_n\cap X\ne\emptyset$. Then $A_{n,X}=A_n\cap X$ and $C_{n,X}=C_n\cap X$ are nonempty disjoint $F_\sigma$-subsets of the normal space  $X$. Applying~\cite[Lemma 1.7]{Vesely} for every $n$, we choose a functionally $G_\delta$-set $D_{n,X}$ in  $X$ such that $A_{n,X}\subseteq D_{n,X}$ and $D_{n,X}\cap C_{n,X}=\emptyset$. Since $X$ is $C$-embedded subspace of  $\upsilon X$, there exists a functionally $G_\delta$-set $D_n$ in $\upsilon X$ such that $D_n\cap X=D_{n,X}$.

Notice that $D_n\cap B_n=\emptyset$, since otherwise $D_n\cap B_n$ is nonempty  $G_\delta$-subset of  $\upsilon X$ which does not intersect  $X$, which contradicts to the $G_\delta$-density of  $X$ in $\upsilon X$. Notice that $A\setminus D_n$ is a $G_{\delta\sigma}$-set in $\upsilon X$ and $(A\setminus D_n)\cap X\subseteq A_{n,X}\setminus D_{n,X}=\emptyset$. Therefore, $A\subseteq D_n$.

Hence, $(D_n)_{n=1}^\infty$ is a sequence of functionally  $G_\delta$-subsets of $\upsilon X$ such that $A\subseteq D_n\subseteq f^{-1}(G_n)$ for every $n$. Then the set $f^{-1}(F)=\bigcap\limits_{n=1}^\infty D_n$ is functionally $G_\delta$ in $\upsilon X$. Therefore, $f\in {\rm K}_1^w(\upsilon X,Y)$.
\end{proof}

Let us observe that the space $X=\{(x_t)_{t\in [0,1]}\in \mathbb R^{[0,1]}:|\{t\in [0,1]: x_t\ne 0\}|\le\aleph_0\}$ is normal~\cite{Corson}, but its Hewitt-Nachbin completion $\upsilon X=\mathbb R^{[0,1]}$ is not normal~\cite{Stone}.

The following examples show that the conditions on $X$ and $Y$ in Proposition~\ref{prop:H1=B1} are essential.

\begin{exm}{\it There exists a Hausdorff compact space $Y$ such that
$$
{\rm K}_1^w(\mathbb R,Y)\setminus {\rm H}_1(\mathbb R,Y)\ne\emptyset.
$$}
\end{exm}

\begin{proof}
 Let $Y=\mathbb R\cup\{\infty\}$ be the Alexandroff one-point compactification of the real line $\mathbb R$ with the discrete topology and let $f:\mathbb R\to Y$ be a mapping such that $f(0)=\infty$ and $f(x)=x$ for all $x\in\mathbb R\setminus\{0\}$. Then the preimage $f^{-1}(F)$ of any functionally closed subset $F\subseteq Y$ is either finite, or has a countable complement, i.e. $f^{-1}(F)$ is $G_\delta$ in $\mathbb R$. Hence, $f\in{\rm K}_1^w(\mathbb R,Y)$. On the other hand, if $V$ is any non-Borel subset of $\mathbb R$, then $V$ is open in $Y$, but $f^{-1}(V)$ is not $F_\sigma$. Therefore, $f\not\in{\rm H}_1(\mathbb R,Y)$.
\end{proof}

\begin{exm} {\it There exists a completely regular space $X$ such that
$$
{\rm H}_1(X,\mathbb R)\setminus {\rm K}_1^w(X,\mathbb R)\ne\emptyset.
$$}
\end{exm}

\begin{proof}
       Let $X$ be the Niemytski plane, i.e. $X=\mathbb R\times [0,+\infty)$, where a base of neighborhoods of $(x,y)\in X$ with $y>0$ form open balls with the center in $(x,y)$, and a base of neighborhoods of $(x,0)$ form the sets $U\cup\{(x,0)\}$, where $U$ is an open ball which tangent to $\mathbb R\times \{0\}$ in the point $(x,0)$.  It is well-known that the space $X$ is perfect and completely regular,  but is not normal.

        Let $E$ be a set which is not of the $G_{\delta\sigma}$-type in $\mathbb R$ and consider the subspace $A=E\times \{0\}$ of $X$. Then $A$ is $F_\sigma$- and $G_\delta$-subset of $X$. We show that $A$ is not functionally $F_\sigma$ in $X$. Assume the contrary and write $A=\bigcup\limits_{n=1}^\infty A_n$, where each $A_n$ is a functionally closed subset of $X$.  For every $n\in\mathbb N$ we take a continuous function $f_n:X\to[0,1]$ such that $A_n=f_n^{-1}(0)$ and for all $(x,y)\in X$ and $m\in\mathbb N$ we put
        $$
        g_{n,m}(x,y)=\left\{\begin{array}{ll}
                              f_n(x,y), & y\ge\frac 1m,\\
                              f_n(x,\frac 1m),  & 0\le y<\frac 1m.
                            \end{array}
        \right.
        $$
        Then $g_{n,m}:\mathbb R\times [0,+\infty)\to [0,1]$ is a continuous function and $\lim\limits_{m\to\infty}g_{n,m}(x,y)=f_n(x,y)$ for every $(x,y)\in X$. Therefore, $f_n\in {\rm B}_1(\mathbb R\times [0,+\infty),[0,1])$ for every $n\in\mathbb N$, which implies that every set $A_n$ is a $G_\delta$-subset of $\mathbb R\times\{0\}$, a contradiction.

 It is easy to see that $\chi_A\in {\rm H}_1(X,\mathbb R)\setminus {\rm K}_1^w(X,\mathbb R)$.
\end{proof}

\section{Approximation Lemma}

\begin{dfn}{\rm
  A family $\mathcal A$ of subsets of topological space $X$ is called {\it strictly functionally discrete} if for every $A\in\mathcal A$ there exists a continuous function $f_A:X\to[0,1]$ such that $A\subseteq f_A^{-1}(0)$ and the family $(f_A^{-1}([0,1)):A\in\mathcal A)$ is discrete.}
\end{dfn}

Evidently, each strictly functionally discrete family is strongly functionally discrete and these two concepts coincide in the case when $\mathcal A$ consists of functionally closed sets.

For a family $\mathcal A=(A_i:i\in I)$ the set $\bigcup\limits_{i\in I}A_i$ is denoted by $\cup\mathcal A$.

\begin{prop}\label{prop:main_technical}
Let $X$ be a topological space, $(Y,d)$ be a metric space and \mbox{$f:X\to Y$} be a mapping. Assume that for some integer $n\ge 2$ there are strictly functionally discrete families $\mathcal F_0$, $\mathcal F_1$,\dots, $\mathcal F_n$ of subsets of $X$ and mappings $p_k:\mathcal F_k\to Y$ for $k\in\{0,\dots,n\}$ such that
\begin{enumerate}
  \item[(A)] the family $\mathcal F_0$ consists of a single set $X$;

  \item[(B)] for every $k\le n$ and $F\in\mathcal F_k$ we get $\sup\limits_{x\in F} d(f(x),p_k(F))<\frac{1}{2^{k+2}}$;

  \item[(C)] for every $k<n$ each set of $\mathcal F_{k+1}$ is contained in some set of $\mathcal F_k$;

  \item[(D)] for every $F\in\mathcal F_1$ the points $p_0(X)$ and $p_1(F)$ can be joined by an arc in $Y$;

  \item[(E)] for any $k\in\{1,\dots,n-1\}$ and sets $F\in\mathcal F_{k+1}$ and $F'\in\mathcal F_k$ with with $F\subseteq F'$ the points $p_{k+1}(F)$ and $p_k(F')$ can be joined by an arc of diameter $<\frac{1}{2^{k+2}}$ in $Y$.
\end{enumerate}
Then there exists a continuous mapping \mbox{$g:X\to Y$} such that for any $k\in\{1,\dots,n\}$ and $x\in \cup {\mathcal F}_k$ we get
\begin{gather}\label{prop:main_technical_ineq}
d(f(x),g(x))<\frac{1}{2^k}.
\end{gather}
\end{prop}

\begin{proof} For every $k\in\{1,\dots,n\}$ and  $F\in\mathcal F_k$ we take a continuous function $\psi_{F,k}:X\to [0,1]$ such that $F\subseteq \psi_{F,k}^{-1}(0)$ and the family $(\psi_{F,k}^{-1}([0,1)):F\in\mathcal F_k)$ is discrete. Put $\varphi_{F,1}=\psi_{F,1}$ for every $F\in\mathcal F_1$. Now for every $F\in\mathcal F_2$ we choose $F'\in \mathcal F_1$ such that $F\subseteq F'$ and set $U_{F,2}=\psi_{F,2}^{-1}([0,1))\cap\varphi_{F',1}^{-1}([0,\frac 12))$. There exists a continuous function $\varphi_{F,2}:X\to [0,1]$ such that $\varphi_{F,2}^{-1}(0)=\psi_{F,2}^{-1}(0)$ and $\varphi_{F,2}^{-1}([0,1))=U_{F,2}$.
Proceeding inductively in this way we obtain a sequence of families $(\varphi_{F,k}:F\in\mathcal F_k)_{1{\le}k{\le}n}$ of continuous functions $\varphi_{F,k}:X\to [0,1]$ such that for any $k\in\{2,\dots,n\}$ and $F\subseteq F_k$ there is $F'\in\mathcal F_{k-1}$ with $F\subseteq F'$ and
  \begin{equation}\label{prop:main_technical_cond4}
   F\subseteq \varphi_{F,k}^{-1}(0)\subseteq\varphi_{F,k}^{-1}([0,1))\subseteq \varphi_{F',k-1}^{-1}([0,\tfrac 12)),
  \end{equation}
and the family $(U_{F,k}:F\in\mathcal F_k)$ is discrete, where $U_{F,k}=\varphi_{F,k}^{-1}([0,1))$.
For every $k$ we put
$$
U_k=\bigcup\limits_{F\in \mathcal F_k} U_{F,k}
$$
and notice that the sets
\begin{gather*}
H_k=\bigcup\limits_{F\in \mathcal F_k} \varphi_{i,k}^{-1}([0,\tfrac 12])\quad \mbox{and}\quad E_k=X\setminus U_k
\end{gather*}
are disjoint and functionally closed in $X$. Hence, there exists a continuous function $h_k:X\to [0,1]$ such that $H_k=h_k^{-1}(1)$ and $E_k=h_k^{-1}(0)$.

For all $x\in X$ let
$$
g_0(x)=p_0(X).
$$
For every $F\in\mathcal F_1$ we take a continuous function $\gamma_{F,1}:[0,1]\to Y$ with $\gamma_{F,1}(0)=p_0(X)$ and $\gamma_{F,1}(1)=p_1(F)$.
Define $g_1:X\to Y$ by the formula
\begin{gather*}
g_1(x)=\left\{\begin{array}{ll}
                 g_0(x), & \mbox{if\,\,} x\in E_1, \\
                \gamma_{F,1}(h_1(x)),& \mbox{if\,\,} x\in \overline{U_{F,1}} \mbox{\,\, for some\,\,} F\in \mathcal F_1.
               \end{array}
  \right.
\end{gather*}
Let us observe that  the family $(U_{F,1}:F\in \mathcal F_1)$ is discrete and for every $x\in \overline{U_{F,1}}\setminus U_{F,1}$ with $F\in \mathcal F_1$ we have $h_{1}(x)=0$ and $g_1(x)=\gamma_{F,1}(0)=g_0(x)$. Moreover, $g_1(x)=p_1(F)$ for $x\in \varphi_{F,1}^{-1}([0,\frac 12])$, $F\in \mathcal F_1$. Therefore,  $g_1$ is well-defined continuous mapping.

Assume that we have already defined continuous mappings $g_1,\dots,g_k:X\to Y$ for $1\le k<n$ such that
$g_k(x)=g_{k-1}(x)$ if $x\in E_{k}$ and $g_k(x)=p_k(F)$ if $x\in \varphi_{F,k}^{-1}([0,\frac 12])$ for some $F\in \mathcal F_k$.
Using (E) for every $F\in\mathcal F_{k+1}$ we choose a set $F'\in\mathcal F_k$ with $F\subseteq F'$ and a continuous arc $\gamma_{F,k+1}:[0,1]\to Y$ of diameter $<\frac{1}{2^{k+2}}$ such that $\gamma_{F,k+1}(0)=p_k(F')$ and $\gamma_{F,k+1}(1)=p_{k+1}(F)$.
Define a continuous mapping $g_{k+1}:X\to Y$ by the formula
$$
g_{k+1}(x)=\left\{\begin{array}{ll}
                g_{k}(x), & \mbox{if\,\,\,}x\in E_{k+1}, \\
                \gamma_{F,k+1}(h_{k+1}(x)),& \mbox{if\,\,} x\in \overline{U_{F,k+1}} \mbox{\,\, for some\,\,} F\in \mathcal F_{k+1}.
              \end{array}
\right.
$$
We prove that  the inequality
\begin{gather}\label{c3}
  d(g_{k+1}(x),g_k(x))<\frac{1}{2^{k+2}}
\end{gather}
holds for every $x\in X$. Clearly, (\ref{c3}) is valid for  $x\in E_{k+1}$. If $x\in U_{F,k+1}$ for $F\in \mathcal F_{k+1}$, then  $x\in \varphi_{F',k}^{-1}([0,\frac 12))$  and
$$
g_{k+1}(x)=\gamma_{F,k+1}(h_{k+1}(x)),\quad g_k(x)=p_k(F')=\gamma_{F,k+1}(0),
$$
which implies~(\ref{c3}).

Proceeding inductively we define continuous mappings  $g_1,\dots,g_n:X\to Y$ which satisfy~(\ref{c3}) for all $k\in\{1,\dots,n-1\}$ and $x\in X$. Moreover, $g_k|_{F}=p_k(F)$ for every $k\in\{1,\dots,n\}$ and $F\in\mathcal F_k$.

We denote $g=g_n$. Let $1\le k\le n$ and $x\in \cup\mathcal F_k$. Then $x\in F$ for some $F\in \mathcal F_k$ and $g_k(x)=p_k(F)$. Taking into consideration~(B), we get
$$
d(f(x),g_k(x))\le \frac{1}{2^{k+2}}.
$$
Then~(\ref{c3}) implies that
\begin{gather*}
  d(f(x),g(x))\le d(f(x),g_k(x))+\sum\limits_{i=k}^{n-1} d(g_i(x),g_{i+1}(x))<\frac{1}{2^{k+2}}+\frac{1}{2^{k+1}}<\frac{1}{2^k},
\end{gather*}
which completes the proof.
\end{proof}

\section{A generalization of the Lebesgue-Hausdorff Theorem}

\begin{thm}\label{Sigma_is_BaireOne}
   Let $X$ be a topological space and $Y$ be a metrizable connected and locally arcwise connected space. Then
   $$
   \Sigma_0^{f*}(X,Y)\subseteq {\rm B}_1(X,Y).
   $$
\end{thm}

\begin{proof} Let $d$ be a metric on $Y$ which generates its topological structure. Using locally arcwise connectedness of $Y$, for every $k\in\mathbb N$ we choose a covering $\mathcal U_k$ of $Y$ by open balls of diameters $<\frac{1}{2^{k+2}}$ such that any two points in any set $U\in\mathcal U_k$
can be joined by an arc of diameter $<\frac{1}{2^{k+2}}$.

Let  $f\in \Sigma_0^{f*}(X,Y)$ and let $\mathcal B$ be a $\sigma$-sfd base for $f$ which consists of functionally closed sets in $X$. For every  $k\in\mathbb N$ we put
$$
\mathcal B_k=(B\in\mathcal B: B\subseteq f^{-1}(U)\,\,\,\mbox{for some}\,\,\,U\in \mathcal U_k).
$$
Then $\mathcal B_k$ is a $\sigma$-sfd family and $X=\cup \mathcal B_k$ for every $k$. According to  \cite[Lemma 13]{Karlova:TA:2015}  for every $k\in\mathbb N$ there exists a sequence  $(\mathcal B_{k,n})_{n=1}^\infty$ of sfd families $\mathcal B_{k,n}=(B_{k,n,i}:i\in I_{k,n})$ of functionally closed subsets of $X$ such that  $\mathcal B_{k,n}\prec \mathcal B_k$, $\mathcal B_{k,n}\prec\mathcal B_{k,n+1}$ for every $n\in\mathbb N$ and $\bigcup\limits_{n=1}^\infty \bigcup\mathcal B_{k,n}=X$. For all $k,n\in\mathbb N$ we put
$$
\mathcal F_{k,n}=(B_{1,n,i_1}\cap\dots\cap B_{k,n,i_k}: i_m\in I_{m,n}, 1\le m\le k)
$$
and notice that every $\mathcal F_{k,n}$ is strongly functionally discrete family of functionally closed sets with
\begin{enumerate}
  \item[(a)] $\mathcal F_{k+1,n}\prec \mathcal F_{k,n}$,

  \item[(b)] $\mathcal F_{k,n}\prec \mathcal F_{k,n+1}$,

  \item[(c)]  $\bigcup\limits_{n=1}^\infty \bigcup\mathcal F_{k,n}=X$.
\end{enumerate}

Fix  $n\ge 2$ and for every $k=1,\dots,n$ we denote $\mathcal F_{k}=\mathcal F_{k,n}$. Take any $y_0\in Y$ and put $p_0(X)=y_0$. For every $F\in\mathcal F_1$ we take $U_F\in\mathcal U_1$ with $f(F)\subseteq U_F$ and any $y_F\in U_F$. Put $p_1(F)=y_F$. Notice that $Y$ is arcwise connected, therefore, the condition (D) of Proposition~\ref{prop:main_technical} holds.

Now for every $F\in \mathcal F_2$ there exist $F'\in \mathcal F_1$ with  $f(F)\subseteq f(F')$ and $U_F\in\mathcal U_2$ with $f(F)\subseteq U_F$. Take any point $y_{F}\in U_{F}\cap U_{F'}$ and put $p_2(F)=y_F$. Notice that we can join $p_2(F)$ and $p_1(F')$ by an arc of diameter $<\frac{1}{2^4}$. Proceeding inductively we obtain mappings $p_k:\mathcal F_k\to Y$ which for every $k\in\{0,\dots,n\}$ satisfy the conditions of Proposition~\ref{prop:main_technical}. Hence, there exists a continuous mapping  $g_n:X\to Y$ such that $d(f(x),g_n(x))<\frac{1}{2^k}$ if $x\in \cup\mathcal F_{k,n}$ for $k\le n$.

We show that the sequence $(g_n)_{n=2}^\infty$ converges pointwise to $f$ on $X$. Let $x\in X$, $\varepsilon>0$ and $k\in\mathbb N$ be such that $\frac{1}{2^k}<\varepsilon$.  Conditions~(b) and (c) imply that there exists  $n_0\ge k$ such that  $x\in \mathcal F_{k,n}$ for all  $n\ge n_0$. Then we have
$$
d(f(x),g_n(x))<\frac{1}{2^k}<\varepsilon
$$
for all  $n\ge n_0$. Hence, $f\in {\rm B}_1(X,Y)$.
\end{proof}

Combining Theorems~\ref{BaireOne_is_SFD} and~\ref{Sigma_is_BaireOne} we get the following result.
\begin{thm}\label{th:EqualityB1SFD}
Let $X$ be a topological space and $Y$ be a metrizable connected and locally arcwise connected space. Then
  $$
  {\rm K}_1(X,Y)\cap\Sigma^f(X,Y)={\rm B}_1(X,Y).
  $$
\end{thm}

\begin{rmk}
  {\rm If $X$ is a normal space, then Proposition~\ref{prop:H1=B1} and Theorem~\ref{th:EqualityB1SFD} imply Theorem~\ref{th:Vesely}.}
\end{rmk}

For a product $Y=\prod\limits_{n=1}^\infty Y_n$ of topological spaces we denote by $\pi_n$ a projection mapping $\pi_n:Y\to Y_n$, $\pi_n(y)=y_n$ for every $y=(y_n)_{n=1}^\infty\in Y$.

\begin{cor}\label{cor:prod}
  Let $X$ be a topological space, $(Y_n)_{n=1}^\infty$ be a sequence of metrizable connected and locally arcwise connected spaces and let $Y\subseteq \prod\limits_{n=1}^\infty Y_n$ be a subspace such that for every $n\in\mathbb N$ there exists a continuous mapping $h_n:\prod\limits_{k=1}^n Y_k\to Y$ with $\pi_n(h_n(y))=y$ for every $y\in \prod\limits_{k=1}^n Y_k$. Then
  $$
  {\rm K}_1(X,Y)\cap\Sigma^f(X,Y)={\rm B}_1(X,Y).
  $$
\end{cor}

\begin{proof}
We only need to prove the inclusion $\Sigma_0^{f*}(X,Y)\subseteq {\rm B}_1(X,Y)$. Let $f\in \Sigma_0^{f*}(X,Y)$. Then $\pi_n\circ f\in B_1(X,Y_n)$ by Theorem~\ref{th:EqualityB1SFD} for every $n$. Therefore, there exists a sequence of continuous mappings $f_{nm}:X\to Y_n$ such that $\lim\limits_{m\to\infty}f_{nm}(x)=\pi_n(f(x))$ for all $x\in X$ and $n\in\mathbb N$. We put $f_n=h_n\circ f_{nn}$. Then $(f_n)_{n=1}^\infty$ is a sequence of continuous mappings $f_n:X\to Y$ which converges to $f$ pointwise on $X$.
\end{proof}

\begin{rmk}
  {\rm Let $Y=\bigcup\limits_{n=1}^\infty \{(y_n)_{n\in\mathbb N}\in[0,1]^\omega: y_n=1 \,\,\,\mbox{and}\,\,\, \forall k>n \,\,\, y_k=0\}$. Then $Y$ is neither arcwise connected nor locally arcwise connected, but the equality ${\rm K}_1(X,Y)={\rm B}_1(X,Y)$ holds for any space $X$ by Corollary~\ref{cor:prod}.}
\end{rmk}

\begin{dfn}
  {\rm A topological space $X$ is called
  \begin{itemize}
    \item   {\it almost arcwise connected} if each pair of nonempty open subsets of $X$ can be joined by an arc in $X$;

    \item  {\it locally almost arcwise connected at $x\in X$} if for any neighborhood $V$ of $x$ there exists a neighborhood $U\subseteq V$ of $x$ such that each pair of nonempty open subsets of $U$ can be joined by an arc in $\overline{V}$;

    \item {\it locally almost arcwise connected} if it is locally almost arcwise connected at every point.
  \end{itemize} }
\end{dfn}

\begin{lem}\label{lemma:aac}
 Let $Y$ be a topological space such that ${\rm H}_1([0,1],Y)\subseteq {\rm B}_1([0,1],Y)$. Then $Y$ is almost arcwise connected.
\end{lem}

\begin{proof}
Consider nonempty open sets $U$ and  $V$ in $Y$ and take different points $u\in U$ and $v\in V$. Let $f(x)=u$ if $x\in [0,1)$ and $f(1)=v$. Then   $f\in{\rm H}_1([0,1],Y)$. By the assumption of the theorem there exists a sequence of continuous mappings  $f_n:[0,1]\to Y$ which is pointwise convergent to $f$. Then there exists a number $n$ such that $f_n(0)\in U$ and $f_n(1)\in V$, which implies almost arcwise connectedness of  $Y$.
\end{proof}

\begin{lem}\label{lemma:loc_aac}
 Let $Y$ be a first-countable topological space such that ${\rm H}_1([0,1],Y)\subseteq {\rm B}_1([0,1],Y)$. Then $Y$ is locally almost arcwise connected.
\end{lem}

\begin{proof}
Suppose to the contrary that there exists a point $y_0\in Y$ and its neighborhood $V$ such that  the condition of locally almost arcwise connectedness at $y_0$ is not satisfied. Let $(V_n:n\in\mathbb N)$ be  a family of all basic neighborhoods of  $y_0$ which are contained in $V$. For every $n\in\mathbb N$ we take nonempty open sets $U_{n,0}$ and $U_{n,1}$ in $V_n$ such that any arc starting in $U_{n,0}$ and ending in $U_{n,1}$ necessarily intersects  $Y\setminus \overline{V}$. For all $n\in\mathbb N$ and $i=0,1$ fix any point $y_{n,i}\in U_{n,i}$. Let $Q=\{[r_{n,0},r_{n,1}]: n\in\mathbb N\}$ be a dense subset of $(\mathbb Q\cap [0,1]^2)$, where all $r_{n,i}$ are distinct. Define a mapping $f:[0,1]\to Y$ by the formula
  $$
  f(x)=\left\{\begin{array}{ll}
                y_0, & \mbox{if\,\,} x\in [0,1]\setminus  Q, \\
                y_{n,i}, & \mbox{if\,\,} x=r_{n,i}\mbox{\,\, for some\,\,} n\in\mathbb N\mbox{\,\,and\,\,} i\in\{0,1\}.
              \end{array}
  \right.
  $$
Notice that  $f([0,1])\subseteq V$ and $f\in{\rm H}_1([0,1],Y)$, since every sequence $(y_{n,i})_{n=1}^\infty$ converges to  $y_0$ in $Y$.  Then there exists a sequence $(f_k)_{k=1}^\infty$ of continuous mappings $f_k:[0,1]\to Y$ which is pointwise convergent to  $f$.

Fix  $m\in\mathbb N$. Consider an open set
$$
G_m=\bigcup\limits_{k\ge m} f_k^{-1}(Y\setminus \overline{V})
$$
and prove that it is dense in $[0,1]$. Let  $n\in\mathbb N$ and $I=[r_{n,0},r_{n,1}]$. Since  $f_k(r_{n,i})\to f(r_{n,i})=y_{n,i}$, there exists $k\ge m$ such that $f_k(r_{n,i})\in U_{n,i}$ for $i=0,1$.  Then $f_k(I)\setminus \overline{V}\ne\emptyset$. It follows that $I\cap G_m\ne\emptyset$.

Notice that $\bigcap\limits_{m=1}^\infty G_m\subseteq f^{-1}(Y\setminus V)=\emptyset$, which contradicts to the Bairness of  $[0,1]$.
\end{proof}

The last two lemmas immediately imply the following result.

\begin{cor}
   Let $Y$ be a first-countable space such that ${\rm H}_1([0,1],Y)\subseteq {\rm B}_1([0,1],Y)$. Then $Y$ is almost arcwise connected and locally almost arcwise connected.
\end{cor}

We observe that the converse proposition is not true even for metrizable separable spaces $Y$ as Example~2 from \cite{Fos} shows.

\begin{lem}\label{lemma:reg_aac_is_lc}
  Every regular locally almost arcwise connected space  $X$ is locally connected.
\end{lem}

\begin{proof}
Fix a point $x\in X$ and an open neighborhood $W$ of $x$. Let $U$ and $V$ are open neighborhoods of $x$ such that $U\subseteq \overline{V}\subseteq W$ and each two open nonempty subsets of $U$ can be joined by an arc in $\overline{V}$. Consider a family $\mathcal C$ which consists of all arcs intersecting $U$ and lying in $W$. Put $G=(\cup \mathcal C)\cup U$ and show that $G$ is connected. Indeed, assume that $G=G_0\cup G_1$, where $G_0$ and $G_1$ are nonempty disjoint relatively open subsets of $G$. By definition of $G$ we have  $U\cap G_0\ne\emptyset$ and $U\cap G_1\ne\emptyset$. Then the sets $U\cap G_0$ and $U\cap G_1$ can be joined by an arc $\gamma$ lying in $W$. Since $\gamma\in\mathcal C$ and the set  $\gamma$ is connected, it follows that $\gamma\subseteq G_0$ or $\gamma\subseteq G_1$, a contradiction.
Hence, $G\subseteq W$ is a connected neighborhood of $x$. Therefore, $X$ is connected im kleinen at $x$.
It remains to apply Theorem 10 from~\cite[p.~90]{Moore}.
\end{proof}

\begin{lem}\label{lemma:complete_aac}
  Any completely metrizable almost arcwise connected and locally almost arcwise connected  space $X$ is arcwise connected and locally arcwise connected.
\end{lem}

\begin{proof}
It is easy to see that $X$ is connected. Lemma~\ref{lemma:reg_aac_is_lc} implies that $X$ is locally connected. Then $X$ is arcwise connected and locally arcwise connected by Theorem 1 from~\cite[p.~254]{Kuratowski:Top:2}.
\end{proof}

\begin{thm} Let $Y$ be a completely metrizable space. Then the following statements are equivalent:
\begin{enumerate}
   \item[{\rm (i)}] $Y$ is  connected and locally arcwise connected;

  \item[{\rm (ii)}] ${\rm K}_1(X,Y)\cap \Sigma^f(X,Y)={\rm B}_1(X,Y)$ for any topological space $X$;

    \item[{\rm(iii)}] ${\rm H}_1([0,1],Y)={\rm B}_1([0,1],Y)$.
\end{enumerate}
\end{thm}

\begin{proof} The implication (i)$\Rightarrow$(ii) follows from Theorem~\ref{th:EqualityB1SFD}.

The implication (ii)$\Rightarrow$(iii) follows from the completeness of $[0,1]$ and Theorem 3 from~\cite{Hansell:1971}.

Finally, Lemmas \ref{lemma:aac},  \ref{lemma:loc_aac} and \ref{lemma:complete_aac} imply the implication (iii)$\Rightarrow$(i).
\end{proof}

According to Theorem 3 from~\cite{Hansell:1971} we have ${\rm H}_1(X,Y)\subseteq \Sigma^f(X,Y)$ for any metrizable space $Y$. Therefore, Theorem~\ref{th:EqualityB1SFD} implies the following fact.

\begin{cor}\label{cor:complete_case}
Let $X$ be a completely metrizable space and $Y$ be a metrizable connected and locally arcwise connected space. Then
  $$
  {\rm H}_1(X,Y)={\rm K}_1(X,Y)={\rm B}_1(X,Y).
  $$
\end{cor}

The following example shows that Corollary~\ref{cor:complete_case} is not valid if $Y$ is not metrizable.

\begin{exm}
{\it There exists a  connected and locally arcwise connected first-countable Lindel\"{o}f space $Y$ such that
$$
{\rm H}_1(\mathbb R,Y)\setminus {\rm B}_1(\mathbb R,Y)\ne\emptyset.
$$
}
\end{exm}

\begin{proof}
  Let $Y$ be the set $(\mathbb R\times [0,+\infty),\tau)$ equipped with a topology $\tau$ such that the sets  $(x-\varepsilon,x+\varepsilon)\times (y-\varepsilon,y+\varepsilon)$, $\varepsilon>0$, are a neighborhood base of a point $(x,y)$ if $y>0$, and the sets  $\{(x,0)\}\cup ((x,x+\varepsilon)\times[0,\varepsilon))$, $\varepsilon>0$, are a neighborhood base of  $(x,y)$ if $y=0$.
  Notice that $Y$ is a Lindel\"{o}f first-countable connected and locally arcwise connected  space. We observe that $Y$ is not regular, since each point $(x,0)$ has not a closed neighborhood base. Evidently, the restriction of $\tau$ to the upper half-plane coincides with the Euclidean topology, and the restriction of  $\tau$ to $\mathbb R\times\{0\}$ coincides with the topology of the Sorgenfrey line.

 Let $\mathbb Q=\{r_n:n\in\mathbb N\}$ and $T=\mathbb R\setminus \mathbb Q$. Define a function $f:\mathbb R\to Y$ by
  $$
  f(x)=\left\{\begin{array}{ll}
  (x-\frac{1}{n},0), & \mbox{if\,\,\,} x=r_n, \\
  (x,0), & \mbox{if\,\,\,}x\in T.
\end{array}
\right.
  $$

In order to show that $f\in {\rm H}_1(\mathbb R,Y)$ it is sufficient to prove that the preimage of a set  $V=[a,b)\times\{0\}$, $a<b$, is  $F_\sigma$ in $\mathbb R$. Denote $E=f^{-1}(V)$. Notice that  $A=[a,b)\setminus E$ and $B=E\setminus (a,b)$ are subsets of $\mathbb Q$ and
\begin{gather}\label{eq:E_Sorg}
  E=([a,b)\setminus A)\cup B.
\end{gather}
Let  $A=\{r_{n_k}:k\in\mathbb N\}$, where $(n_k)_{k=1}^\infty$ is an increasing sequence. Then for every $k$ we have
$$
a\le r_{n_k}\le a+\frac{1}{n_k}.
$$
Hence, $r_{n_k}\to a$ in $\mathbb R$. It follows that $A$ is $G_\delta$ in $\mathbb R$. Then $E$ is $F_\sigma$ in $\mathbb R$ by~(\ref{eq:E_Sorg}).

Now we prove that $f\not\in {\rm B}_1(\mathbb R,Y)$. Assume to the contrary that there exists a sequence of continuous mappings $f_n:\mathbb R\to Y$ which is pointwise convergent to  $f$ on $\mathbb R$. For  $x\in\mathbb R$ we denote
$$
B(x)=\{(x,0)\}\cup((x,+\infty)\times[0,+\infty))
$$
and define
$$
A_n=\bigcap\limits_{k=n}^\infty\{x\in \mathbb R: f_k(x)\in B(f(x))\}
$$
for $n\in\mathbb N$. Then $\bigcup\limits_{n=1}^\infty A_n=\mathbb R$. Since $T$ is a Baire space, there exist $n\in\mathbb N$ and $(a,b)\subseteq\mathbb R$ with
$$
(a,b)\subseteq \overline{A_n\cap T}.
$$
Notice that $A_n\cap T\subseteq F$, where $F=\bigcap\limits_{k=n}^\infty\{x\in \mathbb R: \pi_X(f_k(x))\ge x\}$. Since the projection $\pi_X:Y\to \mathbb R$, $\pi_X(x,y)=x$, is upper semi-continuous, the set $F$ is closed in $\mathbb R$. Then
$(a,b)\subseteq F$. It follows that  $\pi_X(f(x))\ge x$ for all  $x\in (a,b)$, a contradiction.
\end{proof}

\section{Acknowledgements} The author would like to thank the referees for their helpful and constructive comments that greatly contributed to improving the final version of the paper.

\end{document}